\documentclass[a4paper,11pt,twoside]{article}
\RequirePackage[a4paper,head=1cm]{geometry}
\usepackage[utf8x]{inputenc}
\usepackage{textcomp}
\usepackage{amsmath,amsfonts,verbatim,afterpage,euscript,mathrsfs,amssymb,empheq}
\usepackage{amsthm}
\usepackage{dsfont}
\usepackage{hyperref}
\usepackage{pst-fill,pst-grad}
\usepackage{textcomp}
\usepackage{multicol}
\usepackage{amsmath}
\usepackage{amssymb}
\usepackage{hyperref}
\usepackage{dsfont}
\usepackage{colortbl}

\newtheorem{Theoreme}{Theorem}

\newtheorem{Remarque}{\bf Remark}
\newcommand{\mysection}{\setcounter{equation}{0} \section}

\title{Pointwise estimates for rough operators in a metric measure framework under some Ahlfors regularity conditions} 
\author{Diego Chamorro\footnote{Laboratoire de Math\'ematiques et Mod\'elisation d'Evry (LaMME) - UMR 8071. Universit\'e d'Evry Val d'Essonne, 23 Boulevard de France, 91037 Evry Cedex, France. email: \textit{diego.chamorro@univ-evry.fr}}, Anca-Nicoleta Marcoci\footnote{Department of Mathematics and Computer Science. Technical University of Civil Engineering, Bucharest, Bld. Lacul Tei, no. 124, sector 2. Romania. email: \textit{anca.marcoci@utcb.ro}}, Liviu-Gabriel Marcoci\footnote{Department of Mathematics and Computer Science. Technical University of Civil Engineering, Bucharest, Bld. Lacul Tei, no. 124, sector 2. Romania.  email: \textit{liviu.marcoci@utcb.ro}}.} 
\begin{document} 
\maketitle 
\begin{scriptsize}
\abstract{\noindent  We establish a new pointwise estimate for a class of rough operators in the setting of metric measure spaces endowed with a measure which is Ahlfors regular. This pointwise inequality can be divided in two steps: the first one relies in a subrepresentation formula that involves a modified Riesz potential and the upper gradient of the function considered and the second step gives a pointwise control of the Riesz potential in terms of a maximal function and a Morrey norm. We also investigate a family of functional inequalities that can be deduced from this pointwise estimate.}\\

{\footnotesize
\noindent \textbf{MSC (2020) Primary: 42B20; Secondary: 42B25}
}
\end{scriptsize}
\mysection{Introduction}
The main purpose of this article is to study some pointwise estimates for rough operators defined in the general framework of metric measure spaces $(X, d, \mu)$. The analysis of the boundedness of operators in this setting has been developing for several years and many of the classical tools and techniques available in the usual framework of $\mathbb{R}^n$ (endowed with its natural distance and measure) have been generalized or adapted to this more general structure. See for example \cite{EridaniSawano}, \cite{GG}, \cite{Hytonen}, \cite{NTV}, \cite{Samko}, \cite{Tolsa} and the references therein.\\

Since the work of Coifman and Weiss \cite{CW}, it is known that the spaces of homogeneous type provide a general framework in which several results from harmonic analysis on Euclidean spaces
 can be generalized. One of the main ingredients of this theory is related to Borel measures $\mu$ that satisfy the \emph{doubling property} i.e. there exists a positive constant $C$ such that 
\begin{equation}\label{Doubling}
\mu(B(x,2r))\leq C\mu(B(x,r)),
\end{equation}
for every ball $B(x,r)$ of center $x$ and radius $r$. In \cite{NTV} is presented a theory of Calder\'on-Zygmund operators on non-homogeneous spaces by replacing the doubling condition \eqref{Doubling} with the condition 
\begin{equation}\label{n-dim}
\mu(B(x,r))\le Cr^n,
\end{equation}
where $n$ is a positive fixed number and the constant $C$ is independent of $x$ and $r$ in a separable metric space.\\

Let us mention that the doubling property is deeply connected to the Poincar\'e-Sobolev inequality (see  \cite{Heinonen}, \cite{Heinonen0}, \cite{Heinonen1}, \cite{Keith} and \cite{Korobenko}) and all these tools provide an interesting framework to analyze some properties of generic operators that fall outside the usual setting considered in previous works.\\

Motivated by these generalizations, in this article we will study some pointwise estimates between a particular type of operators and some modified Riesz-like potentials. More specifically we will consider the following framework: 
\begin{itemize}
\item[$\bullet$] {\bf Properties of the space}. Let $(X,d,\mu)$ be a metric measure space and we denote $B(x,r)$ the open balls given by  $B(x,r)=\{y\in X: d(x,y)<r\}$. The measure $\mu$ is assumed to be a Borelian measure over the metric space $X$ and we will always assume that the spheres $S(x,r)=\{y\in X: d(x,y)=r\}$ are of measure zero for all $x\in X$ and all $0<r<+\infty$.
\item[$\bullet$] {\bf Properties of the measure}. Besides the general properties of the measure $\mu$, we will consider here a nonnegative, nonatomic measure $\mu$ which satisfies \emph{the upper and lower Ahlfors condition}:
\begin{equation}\label{Ahlfors}
\mathfrak{c}_1r^\nu\leq \mu(B(x,r))\leq \mathfrak{c}_2r^\nu,
\end{equation}
for all $x\in X$ and $0<r<+\infty$, for some power index $0<\nu<+\infty$ and for some positive constants $0<\mathfrak{c}_1\leq \mathfrak{c}_2<+\infty$. This type of measure $\mu$ is called Ahlfors $\nu-$\emph{regular measure} in \cite{Adamowicz}, \cite{Harjulehto}, \cite{Cruz-Shukla}.
\begin{Remarque}
Under this lower and upper Ahlfors-type condition (\ref{Ahlfors}) we obtain that the measure $\mu$ is doubling since, for all $x\in X$ and all $0<r<+\infty$, we have the estimates
$$\mu(B(x,2r))\leq \mathfrak{c}_2 (2r)^\nu= \mathfrak{c}_2 \tfrac{2^\nu}{\mathfrak{c}_1} \; \mathfrak{c}_1 r^\nu\leq  \mathfrak{c}_2 \tfrac{2^\nu}{\mathfrak{c}_1}\, \mu(B(x,r)).$$
The reverse is not true, for example a space with doubling measure which is not Ahlfors regular is the weighted  $\mathbb{R}^n$ with $d\mu=|x|^\alpha dx$, with $\alpha>-n$ (see \cite[Example 3.5, p. 67]{Bjorn} for details).
\end{Remarque}

Note that in the particular case when $X=\mathbb{R}^n$ and $d$ is the usual Euclidean distance, we can easily see that the Lebesgue measure verifies the condition (\ref{Ahlfors}) with $\nu=n$, as we have $|B(x,r)|=v_n r^n$ where $v_n$ is the volume of the $n$-dimensional unit ball.\\

We can now define in a usual way the Lebesgue spaces $L^p(X)$, for $1\leq p\leq +\infty$ as the set of measurable functions $f:X\longrightarrow \mathbb{R}$ such that 
$$\|f\|_{L^p}=\left(\int_{X}|f(x)|^pd\mu(x)\right)^{\frac{1}{p}}<+\infty,$$
with the usual modifications when $p=+\infty$. Recall that if $\frac{1}{p}+\frac{1}{p'}=1$ with $1\leq p\leq +\infty$ and if $f,g:X\longrightarrow \mathbb{R}$ are two measurable functions such that $f\in L^p(X)$ and $g\in L^{p'}(X)$, then we have the classical H\"older inequality given by 
$\displaystyle{\int_{X}|f(x)g(x)|d\mu(x)\leq \|f\|_{L^p}\|g\|_{L^{p'}}}$. We recall also that the set of locally integrable functions $L^1_{loc}(X)$ is given by the condition 
$$\|f\|_{L^1(A)}=\int_A|f(x)|d\mu(x)<+\infty, \quad \mbox{for all compact set } A.$$
We will denote by $f_E$ the integral average defined as 
$$f_E=\frac{1}{\mu(E)}\int_E f(x)d\mu(x),$$
where $E\subset X$ is a $\mu$-measurable set of positive measure $0<\mu(E)<+\infty$.
\item[$\bullet$] {\bf The upper gradient}. In this metric setting, one of the most natural generalization of the gradient of a function was introduced in \cite{Heinonen} as a tool to study quasiconformal maps.  Let $f:X\longrightarrow \mathbb{R}$ be a $L^1_{loc}(X)$ function, the \emph{upper gradient} of $f$ is a function $g:X\longrightarrow [0,+\infty]$ such that we have 
\begin{equation}\label{Def_UpperGradient}
|f(x)-f(y)|\leq \int_{\gamma_{x,y}}g ds,
\end{equation}
for all $x,y\in X$ and for all rectifiable curves $\gamma_{x,y}$ joining $x$ to $y$. See \cite{Heinonen0} and \cite{Heinonen1} for more details and properties of upper gradients. 
\item[$\bullet$] {\bf The Poincaré-Sobolev inequality}.  A metric measure space $(X, d, \mu)$ is said to support a (weak) $(\mathfrak{s},q)$-Poincaré inequality with $1\leq \mathfrak{s}<q<+\infty$, if there exist some constants $1<C<+\infty$ and $1\leq \sigma<+\infty$ such that we have the estimate
\end{itemize}
\begin{equation}\label{Poincare_Sobolev_Ineq}
\left(\frac{1}{\mu(B(x,r))}\int_{B(x,r)}|f(x)-f_B|^q d\mu(x)\right)^{\frac{1}{q}}\leq C r\; \left(\frac{1}{\mu(B(x,\sigma r))}\int_{B(x, \sigma r)}g(x)^\mathfrak{s} d\mu(x)\right)^{\frac{1}{\mathfrak{s}}},
\end{equation}
\begin{itemize}
\item[] whenever $B(x,r)$ is a ball of radius $0<r<+\infty$, $f\in L^1_{loc}(X)$ and $g:X\longrightarrow [0,+\infty]$ is an upper gradient of $f$ in the sense of the expression (\ref{Def_UpperGradient}) above. The definition for spaces supporting Poincaré-Sobolev  inequalities are due to Heinonen and Koskela in \cite{Heinonen}. See also the book \cite{Heinonen1} for more details.

\begin{Remarque}
There is a very deep and strong connection between doubling metric measure spaces and the Poincaré-Sobolev inequality (\ref{Poincare_Sobolev_Ineq}). See \cite{Alvarado}, \cite{Hajlasz}, \cite{Keith} and \cite{Korobenko} for more details on this topic. 
\end{Remarque}
\item[$\bullet$] {\bf The Kernels}. The properties of the operators we want to study here will depend on the properties of their kernels and we will consider a kernel $K:X\times X\setminus\{x=y\} \longrightarrow \mathbb{R}$, which is a measurable function such that, for all $x,y\in X$ (with $x\neq y$), we have the pointwise estimate
\begin{equation}\label{TailleKernel}
|K(x,y)|\leq \frac{C}{d(x,y)^\nu}, 
\end{equation}
where $\nu>0$. We will also assume that, for all $x\neq y$, the limit 
$$\underset{\varepsilon\to 0}{\lim}\int_{\{\varepsilon<d(x,y)\}}K(x,y)d\mu(y),$$
exists for $\mu$-almost every point $x\in X$ and finally, we will assume that we have the following pseudo-radial null condition
\begin{equation}\label{PseudoRadial}
\int_{\{a<d(x,y)<b\}}K(x,y)d\mu(y)=0,
\end{equation}
for all $0<a<b<+\infty$ and for all $x\in X$. Note in particular that we do not impose any regularity condition to the kernel $K(\cdot, \cdot)$ and this type of kernel considered here does not fall in the setting considered in  \cite{GG}, \cite{NTV} or \cite{Tolsa}. Indeed, in these articles the following ``regularity'' condition is assumed for the kernels $K$:
$$|K(s, t)-K(s_0, t)|, |K(t, s)-K(t, s_0)| \leq C\frac{d(s,s_0)^\alpha}{d(t,s_0)^{\nu+\alpha}},$$
for some $\alpha>0$, whenever $d(t, s_0)\geq 2d(s, s_0)$. Instead, we only assume here integrability conditions for the kernels. 
\item[$\bullet$] {\bf The Operators $T_K$ and $T^*_K$}. We will now define, for a suitable function $f:X\longrightarrow \mathbb{R}$, the operator $T_K$,  associated to a kernel $K$ that satisfies the previous conditions (\ref{TailleKernel})-(\ref{PseudoRadial}), by the expression
\begin{equation}\label{Operator}
T_K(f)(x)=\int_X K(x,y)f(y)d\mu(y).
\end{equation}
Associated to this operator we consider the maximal singular operator associated to $T_K$
\begin{equation}\label{MaximalOperator}
T^*_{K}(f)(x)=\underset{\epsilon>0}{\sup}\left|\int_{\{\epsilon<d(x,y)\}}K(x,y)f(y)d\mu(y)\right|.
\end{equation}
The operator $T_K$ defined in (\ref{Operator}) with a kernel $K$ that satisfies the conditions (\ref{TailleKernel})-(\ref{PseudoRadial}) will be denoted here as a \emph{rough Calder\'on-Zygmund} operator. \\
\item[$\bullet$] {\bf The Riesz-type operator}.  If $f:X\longrightarrow \mathbb{R}$ is a locally integrable function then for a parameter $0<\mathfrak{s}<\nu$, we define the following Riesz-type operator $\mathcal{R}_{\mathfrak{s},\mu}$ by the expression
\begin{equation}\label{Def_operatorRieszLike}
\mathcal{R}_{\mathfrak{s},\mu}(f)(x)=\int_{X}\frac{d(x,y)^{\mathfrak{s}}}{\mu(B(x,d(x,y)))}f(y)d\mu(y).
\end{equation}
It is easy to see that if $d\mu$ is the usual Lebesgue measure and if $d(x,y)=|x-y|$ is the natural distance over $\mathbb{R}^n$, we obtain the classical Riesz potentials (up to some dimensional constants): indeed, if we denote by $v_n$ the volume of the $n$-dimensional unit ball, we have
\begin{eqnarray*}
\mathcal{R}_{\mathfrak{s},\mu}(x)&=&\int_{\mathbb{R}^n}\frac{|x-y|^{\mathfrak{s}}}{\mu(B(x,|x-y|))}f(y)d\mu(y)=\int_{\mathbb{R}^n}\frac{|x-y|^\mathfrak{s}}{|B(x,|x-y|)|}f(y)dy\\
&=&\int_{\mathbb{R}^n}\frac{f(y)}{v_n|x-y|^{n-\mathfrak{s}}}dy= c(n)I_\mathfrak{s}(f)(x).
\end{eqnarray*}
\end{itemize}
With these concepts at hand, we will now study some pointwise inequalities that involve the rough maximal operator $T^*_K$ defined in (\ref{MaximalOperator}) and the Riesz-type operator $\mathcal{R}_{\mathfrak{s},\mu}$ presented in (\ref{Def_operatorRieszLike}).\\
\begin{Theoreme}[\bf Rough Operator estimate]\label{Theo_RoughOpEstimate}
Consider $(X, d, \mu)$ a metric measure space endowed with a Ahlfors regular measure $\mu$ that satisfies the lower and upper condition (\ref{Ahlfors}) with some power constant $0<\nu<+\infty$ and with constants $0<\mathfrak{c}_1\leq \mathfrak{c}_2<+\infty$. Assume also that the metric measure space $(X,d,\mu)$ supports the weak Poincaré-Sobolev inequality (\ref{Poincare_Sobolev_Ineq}).\\ 

\noindent For a locally integrable function $f:X\longrightarrow \mathbb{R}$, if $g:X\longrightarrow [0,+\infty]$ denotes its upper gradient in the sense of the expression (\ref{Def_UpperGradient}), (which is assumed to be a locally integrable function) and if we have the condition 
\begin{equation}\label{ConditionConstantes}
2^{1-\nu}\left(\tfrac{\mathfrak{c}_2}{\mathfrak{c}_1}\right)<1,
\end{equation}
then for an operator $T^*_K$ defined in (\ref{MaximalOperator}) associated to a kernel $K(\cdot, \cdot)$ that satisfies the conditions (\ref{TailleKernel})-(\ref{PseudoRadial}), we have the following pointwise estimate
\begin{equation}\label{EstimationOperateur1}
T^*_{K}(f)(x)\leq C \mathcal{R}_{1,\mu}\left(g\right)(x),
\end{equation}
where $\mathcal{R}_{1,\mu}$ is the Riesz-type operator defined in the expression (\ref{Def_operatorRieszLike}) above.
\end{Theoreme}
\noindent Some remarks are in order there. First we note that the estimate (\ref{EstimationOperateur1}) above is a generalization to nonconvolution type Calder\'on-Zygmund operators in a metric measure framework with measures being Ahlfors regular  of our previous work \cite{ChMarcociMarcoci2} and, to the best of our knowledge, this result is new for this class of ``rough'' operators. Remark next that, as pointed out before, we do not impose any Lipschitz-H\"older regularity to the kernels $K(\cdot, \cdot)$ but instead we require the pseudo-radial null condition (\ref{PseudoRadial}). This conditions is reminiscent of the case of convolution type operators as considered in \cite{ChMarcociMarcoci2}, \cite{Hoang}, \cite{Hoang1} or \cite{Hoang2}, where the operators considered there are of the form 
$$T_\Omega(f)(x)=\int_{\mathbb{R}^n}\frac{\Omega(y/|y|)}{|y|^n}f(x-y)dy,$$
and in this case the null condition (\ref{PseudoRadial}) is a straightforward consequence of the fact that the function $\Omega:\mathbb{S}^{n-1}\longrightarrow \mathbb{R}$ is such that $\displaystyle{\int_{\mathbb{S}^{n-1}}\Omega \ d\sigma=0}$. Note now that the condition (\ref{ConditionConstantes}) is essentially technical and imposes a special behavior of the measure $\mu$ and we do not know if it is possible to get rid of this condition. To finish the remarks, we point out that it is an interesting open problem to generalize the estimate (\ref{EstimationOperateur1}) to a metric measure setting where we \emph{only} dispose the upper Ahlfors condition 
$\mu(B(x,r))\leq \mathfrak{c}_2r^\nu$. This research program will probably require a different approach which it is outside the scope of this article.\\

\noindent The estimate (\ref{EstimationOperateur1}) can be transformed into a more useful pointwise inequality and for this we need to introduce two more objects. 
\begin{itemize}
\item[$\bullet$] {\bf  Maximal functions}. For a locally integrable function $f:X\longrightarrow \mathbb{R}$, we define the maximal function $\mathscr{M}_{\mu}$ by the formula
\begin{equation}\label{Def1Maximalfunctions}
\mathscr{M}_{\mu}(f)(x)=\displaystyle{\underset{B \ni x}{\sup } \;\frac{1}{\mu(B)}\int_{B }|f(y)|d\mu(y)},
\end{equation}
where the supremum is taken over all open balls $B$ that contain the point $x$.
\item[$\bullet$] {\bf Morrey spaces}. For $1\leq p\leq q<+\infty$, the Morrey spaces $\mathcal{M}^{p,q}_\mu(X)$ are defined as the set of all measurable functions $f:X\longrightarrow \mathbb{R}$ such that the condition
\begin{equation}\label{Def_Morrey_space}
\|f\|_{\mathcal{M}^{p,q}_{\mu}}=\underset{x\in X, \; r>0}{\sup}\left(\frac{1}{\mu(B(x,r))^{1-\frac{p}{q}}}\int_{B(x,r)}|f(y)|^p d\mu(y)\right)^{\frac{1}{p}}<+\infty,
\end{equation}
is satisfied. Note in particular that if $p=q$, then we recover the usual Lebesgue spaces, \emph{i.e.} we have $L^q(X)=\mathcal{M}^{q,q}_\mu(X)$.
\end{itemize}
Our next  result show explicitly how to control the Riesz-type operator $\mathcal{R}_{\mathfrak{s},\mu}$ by maximal functions and a norm of a Morrey space. 
\begin{Theoreme}[{\bf Morrey-type pointwise inequality}]\label{Theo_Pointwise_Morrey}
Consider $(X, d, \mu)$ a metric measure space endowed with a measure $\mu$ that satisfies the upper and lower Ahlfors condition (\ref{Ahlfors}) with a power index $0<\nu<+\infty$ and constants $0<\mathfrak{c}_1\leq \mathfrak{c}_2<+\infty$.\\ 

\noindent Let $f:X\longrightarrow \mathbb{R}$ be a measurable function that belongs to a Morrey space $\mathcal{M}^{p,q}_\mu(X)$ defined in the expression (\ref{Def_Morrey_space}) above with $1< p\leq q<+\infty$.\\

\noindent If $0<\mathfrak{s}<\nu$ is a parameter such that $\mathfrak{s}<\frac{\nu}{q}$, 
then the Riesz-type operator $\mathcal{R}_{\mathfrak{s},\mu}(f)$ defined in the expression (\ref{Def_operatorRieszLike}) can be controlled pointwise in the following manner:
\begin{equation}\label{PointWiseIneq_Feli}
|\mathcal{R}_{\mathfrak{s},\mu}(f)(x)|\leq C \mathscr{M}_\mu(f)(x)^{1-\frac{q\mathfrak{s}}{\nu}}\|f\|_{\mathcal{M}_\mu^{p,q}}^{\frac{q\mathfrak{s}}{\nu}},
\end{equation}
where $C=C(\mathfrak{s}, \mathfrak{c}_1, \mathfrak{c}_2, \nu, q)$.
\end{Theoreme}
\noindent We first remark here that if $X=\mathbb{R}^n$ (with its natural structure) and if $d\mu=dx$ is the usual Lebesgue measure, then we easily recover some previous results (see \emph{e.g.} \cite{ChMarcociMarcoci2}). Next note that if $p=q>1$ we obtain the estimate $|\mathcal{R}_{\mathfrak{s},\mu}(f)(x)|\leq C \mathscr{M}_\mu(f)(x)^{1-\frac{q\mathfrak{s}}{\nu}}\|f\|_{L^{q}(\mu)}^{\frac{q\mathfrak{s}}{\nu}}$ which is a generalization of the classical Hedberg inequality \cite{Hedberg} to metric measure spaces when the measure is regular in the sense of the Ahlfors condition (\ref{Ahlfors}). See also \cite[Theorem 2.3]{Sihwaningrum} for a similar result in a slightly different framework. Let us also mention that some boundedness properties of quite similar operators were studied in \cite{GG}, \cite{Samko}, and \cite{Sihwaningrum} or \cite{Sihwaningrum1}.\\

The previous result, while interesting in itself, is merely a pretext to obtain a pointwise control over the operators $T^*_K$ (applied to some suitable function $f$) by the maximal function of the upper gradient $g$ of $f$ and a Morrey norm of the upper gradient $g$. More precisely, we have:
\begin{Theoreme}[{\bf A new pointwise inequality}]\label{Theo_Pointwise_OpMorrey}
Consider $(X, d, \mu)$ a metric measure space endowed with a measure $\mu$ that satisfies the Ahlfors condition (\ref{Ahlfors}) with a power index $0<\nu<+\infty$ and constants $0<\mathfrak{c}_1\leq \mathfrak{c}_2<+\infty$ and that supports the weak Poincaré-Sobolev inequality (\ref{Poincare_Sobolev_Ineq}). Assume that $2^{1-\nu}\left(\tfrac{\mathfrak{c}_2}{\mathfrak{c}_1}\right)<1$ and let $f:X\longrightarrow \mathbb{R}$ be a measurable function. If $g:X\longrightarrow [0,+\infty]$ is an upper gradient of $f$ and if $g$ belongs to a Morrey space $\mathcal{M}^{p,q}_\mu(X)$ defined by the expression (\ref{Def_Morrey_space}) above for some parameters $1< p\leq q<+\infty$ such that $\frac{q}{\nu}<1$, then we have the following pointwise estimate
\begin{equation}\label{PointWiseIneq_UpperGradient}
T^*_K(f)(x)\leq C \mathscr{M}_\mu(g)(x)^{1-\frac{q}{\nu}}\|g\|_{\mathcal{M}_\mu^{p,q}}^{\frac{q}{\nu}}.
\end{equation}
\end{Theoreme}
\noindent {\bf Proof.} This result is a consequence of the two previous theorems. Indeed, by the inequality (\ref{EstimationOperateur1}) we write $T^*_{K}(f)(x)\leq C \mathcal{R}_{1,\mu}\left(g\right)(x)$ then, since by the estimate (\ref{PointWiseIneq_Feli}) we have
$$\mathcal{R}_{1,\mu}\left(g\right)(x)\leq C \mathscr{M}_\mu(g)(x)^{1-\frac{q}{\nu}}\|g\|_{\mathcal{M}_\mu^{p,q}}^{\frac{q}{\nu}},$$
we easily derive the pointwise estimate 
$$T^*_K(f)(x)\leq C \mathscr{M}_\mu(g)(x)^{1-\frac{q}{\nu}}\|g\|_{\mathcal{M}_\mu^{p,q}}^{\frac{q}{\nu}},$$
and this ends the proof of the theorem. \hfill $\blacksquare$\\

\noindent As we can see, the pointwise inequality (\ref{PointWiseIneq_UpperGradient}) is a straightforward consequence of the Theorems \ref{Theo_RoughOpEstimate} and \ref{Theo_Pointwise_Morrey}. This control (which is, to the best of our knowledge completely new in this framework) will lead us to some interesting functional inequalities that will be presented in the Section \ref{Secc_Applications} below.\\

The plan of the article is the following. In Section \ref{Secc_Theo_RoughOpEstimate} we present the proof of the Theorem \ref{Theo_RoughOpEstimate}, while in Section \ref{Secc_ProofTheo_Pointwise_Morrey} we prove the Theorem \ref{Theo_Pointwise_Morrey} and, finally, in Section \ref{Secc_Applications} we will present some applications of the previous pointwise estimate (\ref{PointWiseIneq_UpperGradient}). The constants that appear in this paper, such as $C$, may change from one
occurrence to the next.
\mysection{Proof of the Theorem \ref{Theo_RoughOpEstimate}}\label{Secc_Theo_RoughOpEstimate}

Let us start by recalling the definition of the operator $T^*_{K}$ given in the expression  (\ref{MaximalOperator}) above: we thus have
$$T^*_{K}(f)(x)=\underset{\epsilon>0}{\sup}\left|\int_{\{\epsilon<d(x,y)\}}K(x,y)f(y)d\mu(y)\right|,$$
and for some $\epsilon>0$ we consider now the operator
$$T^\epsilon_{K}(f)(x)=\int_{\{\epsilon<d(x,y)\}}K(x,y)f(y)d\mu(y).$$
We recall that we have $T^*_{K}(f)(x)=\underset{\epsilon>0}{\sup}|T^\epsilon_{K}(f)(x)|$ and that $|T_K(f)(x)|\leq T^*_{K}(f)(x)$.\\

Now, for a function $f\in L^1_{loc}(X)$ and for some $k_0\in \mathbb{Z}$ so that $2^{k_0-2}<\epsilon \leq 2^{k_0-1}$, we write
$$T^\epsilon_{K}(f)(x)=\int_{\{\epsilon<d(x,y)\leq 2^{k_0-1}\}}K(x,y)f(y)d\mu(y)+\sum_{k\geq k_0}\int_{\{2^{k-1}<d(x,y)\leq 2^{k}\}}K(x,y)f(y)d\mu(y).$$
If we assume the  pseudo-radially null condition (\ref{PseudoRadial}), we can introduce some constants, that we define later on, in the previous expression to obtain
$$T^\epsilon_{K}(f)(x)=\int_{\{\epsilon<d(x,y)\leq 2^{k_0-1}\}}K(x,y)(f(y)-c_{k_0})d\mu(y)+\sum_{k\geq k_0}\int_{\{2^{k-1}<d(x,y)\leq 2^{k}\}}K(x,y)(f(y)-c_k)d\mu(y),$$
from which we deduce the inequality
$$|T^\epsilon_{K}(f)(x)|\leq\sum_{k\in\mathbb{Z}}\int_{\{2^{k-1}<d(x,y)\leq 2^{k}\}}|K(x,y)(f(y)-c_k)|d\mu(y).$$
Now, since by the hypothesis (\ref{TailleKernel}) we have the control $|K(x,y)|\leq \frac{C}{d(x,y)^\nu}$ for all $x\neq y$, we can write (recalling that we integrate over the set $\{2^{k-1}<d(x,y)\leq 2^{k}\}$)
\begin{eqnarray*}
|T^\epsilon_{K}(f)(x)|&\leq& C\sum_{k\in\mathbb{Z}}\int_{\{2^{k-1}<d(x,y)\leq 2^{k}\}}\frac{1}{d(x,y)^\nu}|(f(y)-c_k)|d\mu(y)\\
&\leq &C\sum_{k\in\mathbb{Z}}\int_{\{2^{k-1}<d(x,y)\leq 2^{k}\}}\frac{1}{(2^{k-1})^\nu}|(f(y)-c_k)|d\mu(y),
\end{eqnarray*}
 which we rewrite as follows 
$$|T^\epsilon_{K}(f)(x)|\leq C 2^\nu\sum_{k\in\mathbb{Z}}\int_{\{d(x,y)\leq 2^{k}\}}\frac{1}{2^{k\nu}}|(f(y)-c_k)|d\mu(y). $$
Now, if we set two parameters $1<\rho, \rho'<+\infty$, where $\rho'$ is such that $\rho'=\frac{\rho}{\rho-1}<\nu$, then by the H\"older inequality with $\frac{1}{\rho}+\frac{1}{\rho'}=1$, we write
$$|T^\epsilon_{K}(f)(x)|\leq C2^\nu \sum_{k\in \mathbb{Z}}\frac{1}{2^{k\nu}} \left(\mu(B(x,2^k))\right)^{\frac{1}{\rho}}\left(\int_{\{d(x,y)\leq 2^{k}\}}|f(y)-c_k|^{\rho'}d\mu(y)\right)^{\frac{1}{\rho'}}.$$
Since the measure $\mu$ is upper Ahlfors regular, we have $\mu(B(x,2^k))\leq \mathfrak{c}_2 2^{k\nu}$ and we obtain
\begin{eqnarray*}
|T^\epsilon_{K}(f)(x)|&\leq &C2^\nu\mathfrak{c}_2^{\frac{1}{\rho}}\sum_{k\in \mathbb{Z}}\frac{1}{2^{k\nu}} 2^{k\frac{\nu}{\rho}}\left(\int_{\{d(x,y)\leq 2^{k}\}}|f(y)-c_k|^{\rho'}d\mu(y)\right)^{\frac{1}{\rho'}}\\
&\leq &C2^\nu\mathfrak{c}_2^{\frac{1}{\rho}}\sum_{k\in \mathbb{Z}}\frac{1}{2^{k\nu(1-\frac{1}{\rho})}}\left(\int_{\{d(x,y)\leq 2^{k}\}}|f(y)-c_k|^{\rho'}d\mu(y)\right)^{\frac{1}{\rho'}},
\end{eqnarray*}
and we have (since $1-\frac{1}{\rho}=\frac{1}{\rho'}$):
\begin{eqnarray*}
|T^\epsilon_{K}(f)(x)|&\leq&C2^\nu\mathfrak{c}_2^{\frac{1}{\rho}}\sum_{k\in \mathbb{Z}}\frac{1}{2^{k\frac{\nu}{\rho'}}}\left(\int_{\{d(x,y)\leq 2^{k}\}}|f(y)-c_k|^{\rho'}d\mu(y)\right)^{\frac{1}{\rho'}}\\
&\leq&C2^\nu\mathfrak{c}_2^{\frac{1}{\rho}}\sum_{k\in \mathbb{Z}}\left(\frac{1}{2^{k\nu}}\int_{\{d(x,y)\leq 2^{k}\}}|f(y)-c_k|^{\rho'}d\mu(y)\right)^{\frac{1}{\rho'}},
\end{eqnarray*}
using again the fact that the measure $\mu$ is upper Ahlfors regular, we have $\frac{1}{ 2^{k\nu}}\leq \mathfrak{c}_2\frac{1}{\mu(B(x,2^k))}$ and we obtain
$$|T^\epsilon_{K}(f)(x)|\leq C2^\nu\mathfrak{c}_2\sum_{k\in \mathbb{Z}}\left(\frac{1}{\mu(B(x,2^k))}\int_{B(x,2^k)}|f(y)-c_k|^{\rho'}d\mu(y)\right)^{\frac{1}{\rho'}}.$$
We fix the constant $c_k=f_{B_k}=\frac{1}{\mu(B(x,2^k))}\displaystyle{\int_{B(x,2^k)}f(y)d\mu(y)}$, so we can write 
$$T^\epsilon_{K}(f)(x)\leq C\sum_{k\in \mathbb{Z}}\left(\frac{1}{\mu(B(x,2^k))}\int_{B(x,2^k)}|f(y)-f_{B_k}|^{\rho'}d\mu(y)\right)^{\frac{1}{\rho'}}.$$
Now we apply the Poincaré-Sobolev inequality given in (\ref{Poincare_Sobolev_Ineq}) in order to obtain (with the values $1=\mathfrak{s}<\rho'<\nu$):
\begin{eqnarray}
|T^\epsilon_{K}(f)(x)|&\leq &C2^\nu\mathfrak{c}_2\sum_{k\in \mathbb{Z}}\left(\frac{1}{\mu(B(x,2^k))}\int_{B(x,2^k)}|f(y)-f_{B_k}|^{\rho'}d\mu(y)\right)^{\frac{1}{\rho'}}\notag\\
&\leq &C2^\nu\mathfrak{c}_2\sum_{k\in \mathbb{Z}}2^k \frac{1}{\mu(B(x,\sigma2^k))}\int_{B(x,\sigma2^k)} g(y)d\mu(y),\label{SumPointWiseInequality}
\end{eqnarray}
where $g$ is an upper gradient of $f$ in the sense of (\ref{Def_UpperGradient}). We study now the sum in the previous formula:
$$\mathscr{S}=\sum_{k\in \mathbb{Z}}2^k\frac{1}{\mu(B(x,\sigma2^k))}\int_{B(x,\sigma2^k)}g(y) d\mu(y),$$
(note that we thus have $|T^\epsilon_{K}(f)(x)|\leq C\mathscr{S}$) which we rewrite as follows
\begin{eqnarray}
\mathscr{S}&\leq &\underbrace{\sum_{k\in \mathbb{Z}}\frac{2^{k}}{\mu(B(x,\sigma2^k))}\int_{\{\sigma2^{k-1}<d(x,y)\leq \sigma2^{k}\}}g(y) d\mu(y)}_{(A)}\notag\\
&&+\underbrace{\sum_{k\in \mathbb{Z}} \frac{2^{k}}{\mu(B(x,\sigma2^k))}\int_{\{d(x,y)\leq \sigma2^{k-1}\}} g(y) d\mu(y)}_{(B)}.\label{FormuleAB}
\end{eqnarray}
We will estimate each of the previous terms separately. 
\begin{itemize}
\item For the term $(A)$ above, we consider first the quantity
$$\frac{2^{k}}{\mu(B(x, \sigma2^k))}\int_{\{\sigma2^{k-1}<d(x,y)\leq \sigma2^{k}\}}g(y)d\mu(y),$$
but since we are working over the sets $\sigma2^{k-1}<d(x,y)\leq \sigma 2^{k}$, we easily deduce the control $\sigma2^{k}\leq Cd(x,y)$ and we can write
\begin{eqnarray*}
\frac{2^{k}}{\mu(B(x, \sigma2^k))}\int_{\{\sigma2^{k-1}<d(x,y)\leq \sigma2^{k}\}}g(y)d\mu(y)\\
\leq \frac{C }{ \sigma\mu(B(x, \sigma2^k))}\int_{\{\sigma2^{k-1}<d(x,y)\leq \sigma2^{k}\}}d(x,y) g(y) d\mu(y).
\end{eqnarray*}
Since over the set $\{\sigma2^{k-1}<d(x,y)\leq \sigma2^{k}\}$ we have $d(x,y)\leq \sigma2^k$, then we obtain $\frac{1}{\mu(B(x, \sigma2^k))}\leq \frac{1}{\mu(B(x, d(x,y)))}$, and we can write
\begin{eqnarray*}
\frac{C}{\mu(B(x, \sigma2^k))}\int_{\{\sigma2^{k-1}<d(x,y)\leq \sigma2^{k}\}}d(x,y)g(y)d\mu(y)\\
\leq  C\int_{\{\sigma2^{k-1}<d(x,y)\leq \sigma2^{k}\}}\frac{d(x,y)}{\mu(B(x, d(x,y)))}g(y)d\mu(y).
\end{eqnarray*}
With this estimate at our disposal, we obtain
\begin{eqnarray*}
(A)&=&\sum_{k\in \mathbb{Z}}\frac{2^{k}}{\mu(B(x,\sigma2^k))}\int_{\{\sigma2^{k-1}<d(x,y)\leq \sigma2^{k}\}}g(y)d\mu(y)\\
&\leq &\frac{C}{ \sigma}\sum_{k\in \mathbb{Z}}\int_{\{\sigma2^{k-1}<d(x,y)\leq \sigma2^{k}\}}\frac{d(x,y)}{\mu(B(x, d(x,y)))}g(y) d\mu(y)\\
&\leq&\frac{C}{ \sigma} \int_{X}\frac{d(x,y)}{\mu(B(x, d(x,y)))}g(y)d\mu(y).
\end{eqnarray*}
Now, recalling the definition of the operator $\mathcal{R}_{1,\mu}$ given in (\ref{Def_operatorRieszLike}) above, we finally obtain the estimate 
$$(A)\leq \frac{C}{ \sigma} \mathcal{R}_{1,\mu}\left(g\right)(x).$$
\item For the term (B) in (\ref{FormuleAB}) we consider the quantity:
$$(B)=\sum_{k\in \mathbb{Z}}\frac{2^{k}}{\mu(B(x,\sigma2^k))}\int_{\{d(x,y)\leq \sigma2^{k-1}\}} g(y)d\mu(y).$$
Due to the upper and lower Ahlfors regularity of the measure $\mu$ given by the conditions (\ref{Ahlfors}) we have
$$\mu(B(x, \sigma2^{k-1}))\leq \mathfrak{c}_2  2^{-\nu}(\sigma 2^k)^\nu\leq \frac{\mathfrak{c}_2}{\mathfrak{c}_1}2^{-\nu}\; \mathfrak{c}_1(\sigma 2^k)^\nu\leq \frac{\mathfrak{c}_2}{\mathfrak{c}_1}2^{-\nu} \mu(B(x, \sigma 2^k)),$$
so we can write
$$\frac{1}{\mu(B(x,\sigma2^k))}\leq  2^{-\nu}\frac{\mathfrak{c}_2}{\mathfrak{c}_1}\frac{1}{\mu( B(x, \sigma 2^{k-1}))},$$
and we have
\begin{eqnarray*}
(B)&=&\sum_{k\in \mathbb{Z}}\frac{2^{k}}{\mu(B(x,\sigma2^k))}\int_{\{d(x,y)\leq \sigma2^{k-1}\}} g(y)d\mu(y)\\
&\leq&\sum_{k\in \mathbb{Z}}2^{k} \; 2^{-\nu}\frac{\mathfrak{c}_2}{\mathfrak{c}_1}\;\frac{1}{\mu( B(x, \sigma2^{k-1})}\int_{\{d(x,y)\leq \sigma2^{k-1}\}} g(y) d\mu(y).
\end{eqnarray*}
We now obtain the estimate
$$(B)\leq\sum_{k\in \mathbb{Z}}2^{1-\nu} \left(\frac{\mathfrak{c}_2}{\mathfrak{c}_1}\right)\;2^{(k-1)}\frac{1}{\mu(B(x, \sigma2^{k-1}))}\int_{\{d(x,y)\leq \sigma2^{k-1}\}} g(y)d\mu(y),$$
which we rewrite as
$$(B)\leq 2^{1-\nu} \left(\frac{\mathfrak{c}_2}{\mathfrak{c}_1}\right) \sum_{k\in \mathbb{Z}}\;2^{(k-1)}\frac{1}{\mu( B(x, \sigma2^{k-1}))}\int_{\{d(x,y)\leq \sigma2^{k-1}\}} g(y)d\mu(y),$$
and we thus have
\begin{eqnarray*}
(B)&\leq &2^{1-\nu} \left(\frac{\mathfrak{c}_2}{\mathfrak{c}_1}\right)\sum_{k\in \mathbb{Z}}\;2^{k}\frac{1}{\mu(B(x, \sigma2^{k}))}\int_{B(x, \sigma2^{k})} g(y) d\mu(y).\\[3mm]
\end{eqnarray*}
\end{itemize}
With these estimates for the terms $(A)$ and $(B)$, getting back to (\ref{FormuleAB}) we can write 
\begin{eqnarray*}
\mathscr{S}&\leq & \frac{C}{\sigma}\mathcal{R}_{1,\mu}\left(g(x)\right) + 2^{1-\nu} \left(\frac{\mathfrak{c}_2}{\mathfrak{c}_1}\right) \sum_{k\in \mathbb{Z}}\;2^{k}\frac{1}{\mu( B(x, \sigma2^{k}))}\int_{B(x, \sigma2^{k})} g(y) d\mu(y)\\
&\leq & \frac{C}{\sigma}\mathcal{R}_{1,\mu}\left(g(x)\right)+ 2^{1-\nu} \left(\frac{\mathfrak{c}_2}{\mathfrak{c}_1}\right)\mathscr{S},
\end{eqnarray*}
but since we have by hypothesis (\ref{ConditionConstantes}) the constraint $2^{1-\nu}\left(\frac{\mathfrak{c}_2}{\mathfrak{c}_1}\right)<1$, we obtain
$$\mathscr{S}\left(1-2^{1-\nu}\left(\frac{\mathfrak{c}_2}{\mathfrak{c}_1}\right)\right)\leq \frac{C}{\sigma} \mathcal{R}_{1,\mu}\left(g\right)(x),$$
which we rewrite as 
$$\mathscr{S}\leq C'\mathcal{R}_{1,\mu}\left(g\right)(x).$$
Thus, coming back to (\ref{SumPointWiseInequality}) we have
$$|T^\epsilon_{K}(f)(x)|\leq C\mathcal{R}_{1,\mu}\left(g\right)(x),$$
and from this estimate (which is uniform in $\epsilon>0$) we deduce the wished control
$$|T^*_{K}(f)(x)|\leq C \mathcal{R}_{1,\mu}\left(g\right)(x),$$
and this ends the proof of the Theorem \ref{Theo_RoughOpEstimate}. \hfill $\blacksquare$
\mysection{Proof of the Theorem \ref{Theo_Pointwise_Morrey}}\label{Secc_ProofTheo_Pointwise_Morrey}
Our starting point is the definition of the operator $\mathcal{R}_{\mathfrak{s},\mu}$ given in the formula (\ref{Def_operatorRieszLike}) above, indeed we have:
$$\mathcal{R}_{\mathfrak{s},\mu}(f)(x)=\int_{X}\frac{d(x,y)^{\mathfrak{s}}}{\mu(B(x,d(x,y)))}f(y)d\mu(y).$$
By considering a positive parameter $\mathcal{K}>0$ that will be fixed later we can write
\begin{eqnarray}
|\mathcal{R}_{\mathfrak{s},\mu}(f)(x)|&\leq&\int_{\{d(x,y)<\mathcal{K}\}}\frac{d(x,y)^{\mathfrak{s}}}{\mu(B(x,d(x,y)))}|f(y)|d\mu(y)+\int_{\{d(x,y)\geq \mathcal{K}\}}\frac{d(x,y)^{\mathfrak{s}}}{\mu(B(x,d(x,y)))}|f(y)|d\mu(y)\notag\\[3mm]
&=&\mathcal{R}_1+\mathcal{R}_2.\label{DeuxIntegrales1}
\end{eqnarray}
The quantity $\mathcal{R}_1$ above is treated in the following manner:
\begin{eqnarray*}
\mathcal{R}_1&=&\int_{\{d(x,y)<\mathcal{K}\}}\frac{d(x,y)^{\mathfrak{s}}}{\mu(B(x,d(x,y)))}|f(y)|d\mu(y)\\
&=&\sum_{j=0}^{+\infty}\int_{\{\mathcal{K}2^{-(j+1)}<d(x,y)<\mathcal{K}2^{-j}\}}\left(\frac{d(x,y)^{\mathfrak{s}}}{\mu(B(x,d(x,y)))}\right)|f(y)|d\mu(y).
\end{eqnarray*}
Noting that over the set $\{\mathcal{K}2^{-(j+1)}<d(x,y)<\mathcal{K}2^{-j}\}$ we have the control 
$$d(x,y)^\mathfrak{s}\leq \mathcal{K}^{ \mathfrak{s}} 2^{-\mathfrak{s} j},$$ 
and since we have the set inclusion $B(x, \mathcal{K}2^{-(j+1)})\subset B(x,d(x,y))$ if $\mathcal{K}2^{-(j+1)}<d(x,y)$, then we can write 
$$\mu(B(x,\mathcal{K} 2^{-(j+1)}))\leq \mu(B(x,d(x,y))),$$ 
and using the lower Ahlfors condition (\ref{Ahlfors}) we obtain
$$\mathfrak{c}_1(\mathcal{K} 2^{-(j+1)})^\nu\leq \mu(B(x,\mathcal{K} 2^{-(j+1)}))\leq \mu(B(x,d(x,y))),$$ 
from which we derive the inequalities
$$\frac{d(x,y)^{\mathfrak{s}}}{\mu(B(x,d(x,y)))}\leq \frac{\mathcal{K}^\mathfrak{s} 2^{-\mathfrak{s} j}}{\mu(B(x,\mathcal{K} 2^{-(j+1)}))}\leq \frac{\mathcal{K}^\mathfrak{s} 2^{-\mathfrak{s} j}}{\mathfrak{c}_1(\mathcal{K} 2^{-(j+1)})^\nu}=\frac{2^{\nu}}{\mathfrak{c}_1}\mathcal{K}^\mathfrak{s} 2^{-\mathfrak{s} j}\frac{1}{(\mathcal{K} 2^{-j})^\nu}.$$
Noting now that by the upper Ahlfors condition (\ref{Ahlfors}) we have the estimate $\frac{1}{(\mathcal{K} 2^{-j})^\nu}\leq \frac{\mathfrak{c}_2}{\mu(B(x,\mathcal{K} 2^{-j}))}$, we thus obtain the inequality
$$\frac{d(x,y)^{\mathfrak{s}}}{\mu(B(x,d(x,y)))}\leq \frac{2^{\nu}}{\mathfrak{c}_1}\mathcal{K}^\mathfrak{s} 2^{-\mathfrak{s} j}\frac{\mathfrak{c}_2}{\mu(B(x,\mathcal{K} 2^{-j}))}.$$
With this control at hand we can then write
\begin{eqnarray*}
\mathcal{R}_1&\leq& \sum_{j=0}^{+\infty} \int_{\{ \mathcal{K} 2^{-(j+1)}<d(x,y)< \mathcal{K}2^{-j}\}}\left(\frac{2^{\nu}}{\mathfrak{c}_1}\mathcal{K}^{\mathfrak{s}} 2^{-\mathfrak{s} j}\frac{\mathfrak{c}_2}{\mu(B(x,\mathcal{K} 2^{-j}))}\right)|f(y)|d\mu(y)\\
&\leq & \frac{2^{\nu}\mathfrak{c}_2}{\mathfrak{c}_1}\sum_{j=0}^{+\infty} \mathcal{K}^\mathfrak{s} 2^{-\mathfrak{s} j} \frac{1}{\mu(B(x,\mathcal{K} 2^{-j}))}\int_{B(x,\mathcal{K} 2^{-j})} |f(y)|d\mu(y)\\
&\leq & \frac{2^{\nu}\mathfrak{c}_2}{\mathfrak{c}_1}\sum_{j=0}^{+\infty} \mathcal{K}^\mathfrak{s} 2^{-\mathfrak{s} j} \mathscr{M}_\mu(f)(x)=\frac{2^{\nu}\mathfrak{c}_2}{\mathfrak{c}_1} \mathcal{K}^\mathfrak{s}  \mathscr{M}_\mu(f)(x)\sum_{j=0}^{+\infty}2^{-\mathfrak{s} j},
\end{eqnarray*}
where in the last estimate we used the definition of the maximal functions $\mathscr{M}_\mu$ given in the expression (\ref{Def1Maximalfunctions}) above. We now obtain the control
from which we easily deduce the following estimate for the first integral of (\ref{DeuxIntegrales1}):
\begin{equation}\label{EstimationI11}
\mathcal{R}_1\leq C \mathcal{K}^\mathfrak{s} \mathscr{M}_\mu(f)(x),
\end{equation}
where $C=C(\mathfrak{s}, \mathfrak{c}_1, \mathfrak{c}_2, \nu)$.\\

Now, for the quantity $\mathcal{R}_2$ in (\ref{DeuxIntegrales1}) we write
$$\mathcal{R}_2=\int_{\{d(x,y)\geq \mathcal{K}\}}\frac{d(x,y)^{\mathfrak{s}}}{\mu(B(x,d(x,y)))}|f(y)|d\mu(y)=\sum_{j=0}^{+\infty}\int_{\{\mathcal{K}2^j\leq d(x,y)\leq \mathcal{K} 2^{j+1}\}}\frac{d(x,y)^{\mathfrak{s}}}{\mu(B(x,d(x,y)))}|f(y)|d\mu(y).$$
Since over the set $\{\mathcal{K}2^j\leq d(x,y)\leq \mathcal{K} 2^{j+1}\}$ we have the control $d(x,y)^{\mathfrak{s}}\leq \mathcal{K}^\mathfrak{s} 2^{(j+1)\mathfrak{s}}$ as well as the set inclusion 
$B(x, \mathcal{K}2^j)\subset B(x, d(x,y))$, we obtain $\mu(B(x, \mathcal{K}2^j))\leq \mu(B(x, d(x,y)))$ and we have 
$$\frac{d(x,y)^{\mathfrak{s}}}{\mu(B(x,d(x,y)))}\leq \frac{\mathcal{K}^\mathfrak{s} 2^{(j+1)\mathfrak{s}}}{\mu(B(x, \mathcal{K}2^j))},$$
so we can write 
\begin{eqnarray*}
\mathcal{R}_2&\leq &\sum_{j=0}^{+\infty} \frac{\mathcal{K}^\mathfrak{s} 2^{(j+1)\mathfrak{s}}}{\mu(B(x,\mathcal{K}2^j))}\int_{\{\mathcal{K}2^j\leq d(x,y)\leq \mathcal{K} 2^{j+1}\}} |f(y)|d\mu(y)\\
&\leq &\sum_{j=0}^{+\infty} \frac{\mathcal{K}^\mathfrak{s} 2^{(j+1)\mathfrak{s}}}{\mu(B(x,\mathcal{K}2^j))}\int_{\{d(x,y)\leq \mathcal{K} 2^{j+1}\}}|f(y)|d\mu(y),
\end{eqnarray*}
which can be rewritten in the following manner
$$\mathcal{R}_2\leq C\mathcal{K}^\mathfrak{s} \sum_{j=0}^{+\infty} 2^{(j+1)\mathfrak{s}}\frac{1}{\mu(B(x,\mathcal{K}2^j))}\int_{B(x,\mathcal{K} 2^{j+1})}|f(y)|d\mu(y).$$
Now, with the H\"older inequality with $\frac{1}{p}+\frac{1}{p'}=1$ and $1< p<+\infty$, we have
\begin{eqnarray*}
\mathcal{R}_2&\leq &C\mathcal{K}^\mathfrak{s} \sum_{j=0}^{+\infty} 2^{(j+1)\sigma} \frac{1}{\mu(B(x,\mathcal{K}2^j))}\left(\int_{B(x, \mathcal{K} 2^{j+1})}|f(y)|^p d\mu(y)\right)^{\frac{1}{p}}\left(\int_{B(x, \mathcal{K} 2^{j+1})}d\mu(y)\right)^{\frac{1}{p'}}\\
&\leq &C\mathcal{K}^\mathfrak{s} \sum_{j=0}^{+\infty} 2^{(j+1)\mathfrak{s}} \frac{1}{\mu(B(x,\mathcal{K} 2^j))}\left(\int_{B(x, \mathcal{K} 2^{j+1})}|f(y)|^p d\mu(y)\right)^{\frac{1}{p}}\mu(B(x, \mathcal{K} 2^{j+1}))^{\frac{1}{p'}}.
\end{eqnarray*}
Now, we rewrite the previous expression as follows (where $1< p\leq q<+\infty$):
\begin{eqnarray*}
\mathcal{R}_2&\leq &C\mathcal{K}^\mathfrak{s} \sum_{j=0}^{+\infty} 2^{(j+1)\mathfrak{s}} \frac{\mu(B(x, \mathcal{K} 2^{j+1}))^{\frac{1}{p'}}}{\mu(B(x,\mathcal{K}2^j))}\\
&&\times \mu(B(x, \mathcal{K} 2^{j+1}))^{\frac{1}{p}-\frac{1}{q}}\left(\frac{1}{\mu(B(x, \mathcal{K} 2^{j+1}))^{1-\frac{p}{q}}}\int_{B(x, \mathcal{K} 2^{j+1})}|f(y)|^p d\mu(y)\right)^{\frac{1}{p}}\\
&\leq &C\mathcal{K}^\mathfrak{s} \sum_{j=0}^{+\infty} 2^{(j+1)\mathfrak{s}} \frac{\mu(B(x, \mathcal{K} 2^{j+1}))^{1-\frac{1}{q}}}{\mu(B(x,\mathcal{K} 2^j))}\left(\frac{1}{\mu(B(x, \mathcal{K} 2^{j+1}))^{1-\frac{p}{q}}}\int_{B(x, \mathcal{K} 2^{j+1})}|f(y)|^p d\mu(y)\right)^{\frac{1}{p}},
\end{eqnarray*}
and thus, using the definition of the Morrey spaces $\mathcal{M}_\mu^{p,q}(X)$ given in (\ref{Def_Morrey_space}), we obtain
$$\mathcal{R}_2\leq C\mathcal{K}^\mathfrak{s} \|f\|_{\mathcal{M}_\mu^{p,q}}\sum_{j=0}^{+\infty} 2^{(j+1)\mathfrak{s}} \frac{\mu(B(x, \mathcal{K} 2^{j+1}))^{1-\frac{1}{q}}}{\mu(B(x,\mathcal{K}2^j))}.$$
Using again the upper and lower Ahlfors condition (\ref{Ahlfors}), we observe that we have
$$\frac{\mu(B(x, \mathcal{K} 2^{j+1}))^{1-\frac{1}{q}}}{\mu(B(x,\mathcal{K}2^j))}\leq \frac{\mathfrak{c}_2 ( \mathcal{K} 2^{j+1})^{\nu(1-\frac{1}{q})}}{\mathfrak{c}_1(\mathcal{K} 2^{j})^\nu}=\frac{\mathfrak{c}_2}{\mathfrak{c}_1}\mathcal{K}^{-\frac{\nu}{q}}2^{\nu(1-\frac{1}{q})}2^{-j\frac{\nu}{q}},$$
so we can write 
\begin{eqnarray*}
\mathcal{R}_2&\leq &C\mathcal{K}^\mathfrak{s} \|f\|_{\mathcal{M}_\mu^{p,q}}\sum_{j=0}^{+\infty}2^{(j+1)\mathfrak{s}} \frac{\mu(B(x, \mathcal{K} 2^{j+1}))^{1-\frac{1}{q}}}{\mu(B(x,\mathcal{K}2^j))}\\
&\leq &  C\mathcal{K}^\mathfrak{s} \|f\|_{\mathcal{M}_\mu^{p,q}}\frac{\mathfrak{c}_2}{\mathfrak{c}_1}\mathcal{K}^{-\frac{\nu}{q}}2^{\nu(1-\frac{1}{q})}\sum_{j=0}^{+\infty} 2^{-j\frac{\nu}{q}+(j+1)\mathfrak{s}}=C\mathcal{K}^{\mathfrak{s}-\frac{\nu}{q}} \|f\|_{\mathcal{M}_\mu^{p,q}}\frac{\mathfrak{c}_2}{\mathfrak{c}_1}2^{\nu(1-\frac{1}{q})+\mathfrak{s}}\sum_{j=0}^{+\infty} 2^{-j(\frac{\nu}{q}-\mathfrak{s})},
\end{eqnarray*}
and since the previous sum converges (recall that we have $\mathfrak{s}<\frac{\nu}{q}$)  we finally obtain
\begin{equation}\label{EstimationI12}
\mathcal{R}_2\leq  C'\mathcal{K}^{\mathfrak{s}-\frac{\nu}{q}} \|f\|_{\mathcal{M}_\mu^{p,q}},
\end{equation}
with $C'=C(\mathfrak{s}, \mathfrak{c}_1, \mathfrak{c}_2, \nu,q)$.\\

With the estimates (\ref{EstimationI11}) and (\ref{EstimationI12}) at hand, we can come back to the expression (\ref{DeuxIntegrales1}) and we obtain
$$|\mathcal{R}_{\mathfrak{s},\mu}(f)(x)|\leq \mathcal{R}_1+\mathcal{R}_2\leq C \mathcal{K}^\mathfrak{s} \mathscr{M}_\mu(f)(x)+C' \mathcal{K}^{\mathfrak{s}-\frac{\nu}{q}}\|f\|_{\mathcal{M}_\mu^{p,q}}.$$
To continue, we set $\mathcal{K}=\left(\frac{\mathscr{M}_\mu(f)(x)}{\|f\|_{\mathcal{M}_\mu^{p,q}}}\right)^{-\frac{q}{\nu}}$, and we have 
$$|\mathcal{R}_{\mathfrak{s},\mu}(f)(x)|\leq C \mathscr{M}_\mu(f)(x)^{1-\frac{q\mathfrak{s}}{\nu}}\|f\|_{\mathcal{M}_\mu^{p,q}}^{\frac{q\mathfrak{s}}{\nu}},$$
which is the wished inequality and this ends the proof of the Theorem \ref{Theo_Pointwise_Morrey}.\hfill $\blacksquare$
\mysection{Functional inequalities}\label{Secc_Applications}
Once we have at our disposal the pointwise estimate (\ref{PointWiseIneq_UpperGradient}) it is a \emph{relatively} easy task to deduce a family of functional inequalities. First we consider  for some real parameter $\mathfrak{r}$, a functional space $(E^\mathfrak{r}, \|\cdot\|_{E^\mathfrak{r}})$ given by the condition
$$E^\mathfrak{r}(X)=\{f:X\longrightarrow \mathbb{R}, f \mbox{ is measurable and } \|f\|_{E^\mathfrak{r}}<+\infty\},$$
where $\|\cdot\|_{E^\mathfrak{r}}$ is a ``norm'' which depends on the parameter $\mathfrak{r}$. We can assume for simplicity that $1<\mathfrak{r}<+\infty$ - the typical example of such space is the Lebesgue space $L^\mathfrak{r}(X)$ endowed with the usual norm $\|\cdot\|_{L^\mathfrak{r}}$. If the quantity $\|\cdot\|_{E^\mathfrak{r}}$ satisfies the following conditions: \\
\begin{itemize}
\item if we have  two measurable functions $f,g:X\longrightarrow \mathbb{R}$ with the  (\emph{a.e.}) pointwise inequality $|f(x)|\leq |g(x)|$ valid, such that $f,g\in E^{\mathfrak{r}}(X)$ then 
\begin{equation}\label{Monotony}
\|f\|_{E^\mathfrak{r}}\leq \|g\|_{E^\mathfrak{r}}.
\end{equation}
\item for some real positive power $\rho$, such that $\rho\, \mathfrak{r}>1$, we have the identity 
\begin{equation}\label{Puissance}
\||f|^\rho\|_{E^\mathfrak{r}}=\|f\|_{E^{\rho\mathfrak{r}}}^\rho,
\end{equation}
whenever $f\in E^{\rho\mathfrak{r}}(X)$. Note in particular that for the usual Lebesgue spaces we have $\||f|^\rho\|_{L^\mathfrak{r}}=\|f\|_{L^{\rho\mathfrak{r}}}^\rho$ and also for the Morrey spaces given in (\ref{Def_Morrey_space}) we have the identity $\||f|^\rho\|_{\mathcal{M}^{p,q}_\mu}=\|f\|_{\mathcal{M}^{\rho p, \rho q}_\mu}^\rho$.
\item for a generic maximal operator $\mathscr{M}_\mu$ we have the boundedness  in the spaces $E^{\mathfrak{r}}(X)$. 
\begin{equation}\label{GeneralMaximal}
\|\mathscr{M}_\mu(f)\|_{E^\mathfrak{r}}\leq C\|f\|_{E^{\mathfrak{r}}}.
\end{equation}
\end{itemize}
Now if we consider the generic pointwise estimate
$$T^*_K(f)(x)\leq C \mathscr{M}_\mu(g)(x)^{1-\frac{q}{\nu}}\|g\|_{\mathcal{M}_\mu^{p,q}}^{\frac{q}{\nu}},$$
(where the indexes $ \nu, p, q$ are given in the context of the Theorem \ref{Theo_Pointwise_OpMorrey})
we can use the property (\ref{Monotony}) to obtain: 
$$\|T^*_K(f)\|_{E^\mathfrak{r}}\leq C \left\|\mathscr{M}_\mu(g)^{1-\frac{q}{\nu}}\right\|_{E^\mathfrak{r}}\|g\|_{\mathcal{M}^{p,q}_\mu}^{\frac{q}{\nu}},$$
and then by the property (\ref{Puissance}) for the space $E^\mathfrak{r}(X)$ we have 
$$\|T^*_K(f)\|_{E^\mathfrak{r}}\leq C\left\|\mathscr{M}_\mu(g)\right\|_{E^{(1-\frac{q}{\nu})\mathfrak{r}}}^{1-\frac{q}{\nu}}\|g\|_{\mathcal{M}^{p,q}_\mu}^{\frac{q}{\nu}}.$$
Next, if the maximal operator $\mathscr{M}_\mu$ is bounded in the space $E^{(1-\frac{q}{\nu})\mathfrak{r}}(X)$ then by the property (\ref{GeneralMaximal}) we have that

\begin{equation}\label{GeneralFuncIneq}
\|T^*_K(f)\|_{E^\mathfrak{r}}\leq C\|g\|_{E^{(1-\frac{q}{\nu})\mathfrak{r}}}^{1-\frac{q}{\nu}}\|g\|_{\mathcal{M}^{p,q}_\mu}^{\frac{q}{\nu}},
\end{equation}
which is the functional inequality that we want to establish (of course, as long as we have that the upper gradient $g$ satisfies $g\in E^{(1-\frac{q}{\nu})\mathfrak{r}}(X)$ and $g\in \mathcal{M}^{p,q}_\mu(X)$). Let us insist that we did not specify in the inequality above the relationship between the indexes $\mathfrak{r}$, $\nu$ or $p,q$ and thus the previous functional inequality must be understood as a ``theoretical'' and fairly general estimate: the relationship between these indexes must be understood in the setting of the Theorem \ref{Theo_Pointwise_OpMorrey} and also must be related to the properties of the functional framework given by the space $E^{\mathfrak{r}}(X)$.\\

Let us remark now that when dealing with norms over functional spaces, it is quite natural to ask for the condition (\ref{Monotony}), however the property (\ref{Puissance}) is far less natural and in some situations the boundedness of the maximal function $\mathscr{M}_\mu$ given in (\ref{GeneralMaximal}) just fails (think for example in the space $L^1(\mathbb{R}^n)$ where we \emph{do not} have $\mathscr{M}_\mu: L^1\to L^1$).\\

In what follows we study some functional inequalities that can be deduced from the pointwise estimate given in the expression (\ref{PointWiseIneq_UpperGradient}).
\begin{itemize}
\item {\bf Lebesgue spaces $L^\mathfrak{r}(X)$}. Note that for, $1<\mathfrak{r}<+\infty$, the norm 
$$\|f\|_{L^\mathfrak{r}}=\left(\int_{X}|f(x)|^\mathfrak{r}d\mu(x)\right)^{\frac{1}{\mathfrak{r}}},$$
obviously satisfies the property (\ref{Monotony}) and for the property (\ref{Puissance}), as announced before, we simply have $\displaystyle{\||f|^\rho\|_{L^{\mathfrak{r}}}=\left(\int_{X}|f(x)|^{\rho\mathfrak{r}}d\mu(x)\right)^{\frac{1}{\mathfrak{r}}}=\|f\|_{L^{\rho\mathfrak{r}}}^\rho}$. Moreover, following \cite{CW}, \cite{NTV1}, the maximal function $\mathscr{M}_\mu$ associated to the measure $\mu$ and defined in the expression (\ref{Def1Maximalfunctions}) above, satisfies the control 
$$\|\mathscr{M}_\mu(f)\|_{L^{\mathfrak{r}}}\leq C\|f\|_{L^{\mathfrak{r}}},$$
as long as $1<\mathfrak{r}< +\infty$. Under the hypotheses of the Theorem \ref{Theo_Pointwise_OpMorrey} (\emph{i.e.}  $f:X\longrightarrow \mathbb{R}$ is a measurable function, $g:X\longrightarrow [0,+\infty[$ is an upper gradient of $f$ and $g$ is such that $g\in\mathcal{M}^{p,q}_\mu(X)$ with $1<p\leq q<+\infty$ and  $1<\frac{\nu}{q}$, $1<p\leq q <+\infty$) if $1<\mathfrak{r}<+\infty$ is an index such that $(1-\frac{q}{\nu})\mathfrak{r}>1$, then setting $E^\mathfrak{r}(X)=L^\mathfrak{r}(X)$ in (\ref{GeneralFuncIneq}), we obtain
\begin{equation}\label{InegaliteLebesgue}
\|T^*_K(f)\|_{L^\mathfrak{r}}\leq C\|g\|_{L^{(1-\frac{q}{\nu})\mathfrak{r}}}^{1-\frac{q}{\nu}}\|g\|_{\mathcal{M}^{p,q}_\mu}^{\frac{q}{\nu}}.
\end{equation}
To the best of our knowledge the previous estimation (\ref{InegaliteLebesgue}) seems to be new in this setting.\\

Remark in particular that if we set $\mathfrak{r}=\frac{q\nu}{\nu-q}$ and $p=q$, since we have the space identification $\mathcal{M}_\mu^{q,q}(X)=L^q(X)$, we easily derive from the previous estimate the following inequality 
$$\|T^*_K(f)\|_{L^\mathfrak{r}}\leq C\|g\|_{L^{q}}^{1-\frac{q}{\nu}}\|g\|_{\mathcal{M}^{q,q}_\mu}^{\frac{q}{\nu}}=C\|g\|_{L^{q}}^{1-\frac{q}{\nu}}\|g\|_{L^{q}}^{\frac{q}{\nu}},$$
\emph{i.e.}
$$\|T^*_K(f)\|_{L^\mathfrak{r}}\leq C\|g\|_{L^{q}},$$
and this gives a new Sobolev-like inequality. 
\item {\bf Lorentz spaces $L^{\mathfrak{r},\mathfrak{m}}(X)$}. Lorentz spaces are an useful generalization of Lebesgue spaces. For a measurable function $f:X\longrightarrow \mathbb{R}$, for $\alpha\geq 0$ and for some indexes $1\leq \mathfrak{r}<+\infty$ and $1\leq \mathfrak{m}< +\infty$, we can characterize the Lorentz spaces $L^{\mathfrak{r},\mathfrak{m}}(X)$ by the condition 
$$\|f\|_{L^{\mathfrak{r},\mathfrak{m}}}=\mathfrak{r}^{\frac{1}{\mathfrak{m}}}\left(\int_{0}^{+\infty}\left(\alpha\, \mu\left(\{x\in X: |f(x)|>\alpha\} \right)^{\frac{1}{\mathfrak{r}}}\right)^{\mathfrak{m}}\frac{d\alpha}{\alpha}\right)^{\frac {1}{\mathfrak{m}}}<+\infty,$$
and when $\mathfrak{m}=+\infty$, we simply write $\|f\|_{L^{\mathfrak{r},\infty}}=\underset{\alpha >0}{\sup} \; \left\{\alpha \mu\left(\{x\in X: |f(x)|>\alpha\} \right)^{\frac{1}{\mathfrak{r}}}\right\}$.
For these spaces it is easy to see that we have the property (\ref{Monotony}) as well as the property (\ref{Puissance}): for some index $\rho>0$ and for $1\leq \mathfrak{r}<+\infty$, $1\leq\mathfrak{m}\leq +\infty$, we have 
$$\||f|^\rho\|_{L^{\mathfrak{r},\mathfrak{m}}}=\|f\|_{L^{\rho\mathfrak{r}, \rho\mathfrak{m}}}^\rho,$$
see \cite[Proposition 1.2.11]{Chamorro} for a proof. The boundedness of the maximal function $\mathscr{M}_\mu$ in this context was studied in \cite{CW}, \cite{NTV1}, see also \cite{Kos}.\\

Now, in the framework of the Theorem \ref{Theo_Pointwise_OpMorrey}, if $f:X\longrightarrow \mathbb{R}$ is a measurable function and if $g:X\longrightarrow [0,\infty)$ is an upper gradient of $f$, then by applying the Lorentz norm $L^{\mathfrak{r}, \mathfrak{m}}$ to the pointwise inequality (\ref{PointWiseIneq_UpperGradient}) we easily derive the functional control 
\begin{eqnarray*}\label{InegaliteLorentz}
\|T^*_K(f)\|_{L^{\mathfrak{r}, \mathfrak{m}}}&\leq& C\|\mathscr{M}_\mu(g)\|_{L^{(1-\frac{q}{\nu})\mathfrak{r}, (1-\frac{q}{\nu})\mathfrak{m}}}^{1-\frac{q}{\nu}}\|g\|_{\mathcal{M}^{p,q}_\mu}^{\frac{q}{\nu}}\\
&\leq& C\|g\|_{L^{(1-\frac{q}{\nu})\mathfrak{r}, (1-\frac{q}{\nu})\mathfrak{m}}}^{1-\frac{q}{\nu}}\|g\|_{\mathcal{M}^{p,q}_\mu}^{\frac{q}{\nu}},
\end{eqnarray*}
where we have $\frac{q}{\nu}<1$, $1<(1-\frac{q}{\nu})\mathfrak{r}<+\infty$ and $1<(1-\frac{q}{\nu})\mathfrak{m}<+\infty$.  The boundedness of the maximal function $\mathscr{M}_\mu$ in the  context of measure spaces was studied in \cite{CW}, \cite{NTV1}, see also \cite{Kos}.\\

A particular case of the previous inequality is the following: if $1<\mathfrak{r}<+\infty$ is such that $(1-\frac{q}{\nu})\mathfrak{r}=1$, then we can write, by taking the $L^{\mathfrak{r}, \infty}$ norm to the both sides of the pointwise control (\ref{PointWiseIneq_UpperGradient}):
$$\|T^*_K(f)\|_{L^{\mathfrak{r}, \infty}}\leq C\|\mathscr{M}_\mu(g)\|_{L^{(1-\frac{q}{\nu})\mathfrak{r}, \infty}}^{1-\frac{q}{\nu}}\|g\|_{\mathcal{M}^{p,q}_\mu}^{\frac{q}{\nu}}=C\|\mathscr{M}_\mu(g)\|_{L^{1, \infty}}^{1-\frac{q}{\nu}}\|g\|_{\mathcal{M}^{p,q}_\mu}^{\frac{q}{\nu}},$$
and since the maximal function $\mathscr{M}_\mu$ is bounded from $L^1(X)$ to $L^{1,\infty}(X)$, we obtain the estimate
\begin{equation*}
\|T^*_K(f)\|_{L^{\mathfrak{r}, \infty}}\leq C\|\mathscr{M}_\mu(g)\|_{L^{1, \infty}}^{1-\frac{q}{\nu}}\|g\|_{\mathcal{M}^{p,q}_\mu}^{\frac{q}{\nu}}\leq C\|g\|_{L^{1}}^{1-\frac{q}{\nu}}\|g\|_{\mathcal{M}^{p,q}_\mu}^{\frac{q}{\nu}}.
\end{equation*}
Note that, in the scale of Lebesgue and Lorentz spaces, the case when $(1-\frac{q}{\nu})\mathfrak{r}=1$ seems to be a lower end point for the parameter $1<\mathfrak{r}<+\infty$.
\item {\bf Morrey spaces $\mathcal{M}^{p,q}_\mu$}. These spaces were introduced with the condition (\ref{Def_Morrey_space}) above. As before, the property (\ref{Monotony}) is straightforward and for $\rho>0$ we can write (as long as $\rho p>1$)
\begin{eqnarray*}
\||f|^\rho\|_{\mathcal{M}^{p,q}_\mu}&=&\underset{x\in X, \; r>0}{\sup}\left(\frac{1}{\mu(B(x,r))^{1-\frac{p}{q}}}\int_{B(x,r)}(|f(y)|^\rho)^{p} d\mu(y)\right)^{\frac{1}{p}}\\
&=&\underset{x\in X, \; r>0}{\sup}\left(\frac{1}{\mu(B(x,r))^{1-\frac{(\rho p)}{(\rho q)}}}\int_{B(x,r)}|f(y)|^{\rho p} d\mu(y)\right)^{\frac{1}{p}}=\|f\|_{\mathcal{M}^{\rho p, \rho q}_\mu}^\rho,
\end{eqnarray*}
which is the property (\ref{Puissance}). The boundedness of the maximal function $\mathscr{M}_\mu$ in the setting of Morrey spaces was studied in \cite{Sihwaningrum} and \cite{Sihwaningrum1} (see also the references therein).\\

In the setting of the Theorem \ref{Theo_Pointwise_OpMorrey}, if $f:X\longrightarrow \mathbb{R}$ is a measurable function and if $g:X\longrightarrow [0,+\infty]$ is an upper gradient of $f$, we can obtain the following inequality
\begin{eqnarray*}\label{InegaliteMorrey}
\|T^*_K(f)\|_{\mathcal{M}_\mu^{p_1, q_1}}&\leq & C\|\mathscr{M}_\mu(g)\|_{\mathcal{M}_\mu^{(1-\frac{q}{\nu})p_1, (1-\frac{q}{\nu})q_1}}^{1-\frac{q}{\nu}}\|g\|_{\mathcal{M}^{p,q}_\mu}^{\frac{q}{\nu}}\\
&\leq&C\|g\|_{\mathcal{M}_\mu^{(1-\frac{q}{\nu})p_1, (1-\frac{q}{\nu})q_1}}^{1-\frac{q}{\nu}}\|g\|_{\mathcal{M}^{p,q}_\mu}^{\frac{q}{\nu}},
\end{eqnarray*}
where we have $\frac{q}{\nu}<1$, $1<p\leq q<+\infty$ and $1<p_1\leq q_1<+\infty$, $1<(1-\frac{q}{\nu})p_1<+\infty$ and $1<(1-\frac{q}{\nu})q_1<+\infty$.\\

This estimate seems to be new in this context. Note that one particular case of the previous inequality is given when $(1-\frac{q}{\nu})p_1=p$ and $(1-\frac{q}{\nu})q_1=q$: we can then deduce from (\ref{InegaliteMorrey}), the control 
\begin{eqnarray*}
\|T^*_K(f)\|_{\mathcal{M}_\mu^{p_1, q_1}}&\leq &C\|g\|_{\mathcal{M}_\mu^{(1-\frac{q}{\nu})p_1, (1-\frac{q}{\nu})q_1}}^{1-\frac{q}{\nu}}\|g\|_{\mathcal{M}^{p,q}_\mu}^{\frac{q}{\nu}}=C\|g\|_{\mathcal{M}_\mu^{p, q}}^{1-\frac{q}{\nu}}\|g\|_{\mathcal{M}^{p,q}_\mu}^{\frac{q}{\nu}}\\
&\leq & C\|g\|_{\mathcal{M}_\mu^{p, q}}.
\end{eqnarray*}
\item {\bf Orlicz spaces}. 
Let us recall that if $\phi:[0,+\infty[\longrightarrow[0,+\infty]$ is a left-continuous increasing function with $\phi(0)=0$ which is neither identically zero nor identically infinite on $]0,+\infty[$, we can consider the corresponding \emph{Young function} $\Phi(t)=\displaystyle{\int_{0}^t\phi(\tau)d\tau}$ and then the Orlicz space $L^\Phi(X)$ associated to the function $\Phi$ is defined as the set of $\mu$ measurable functions $f:X\longrightarrow \mathbb{R}$ such that the Luxemburg norm
\begin{equation*}
\|f\|_{L^\Phi}=\inf\left\{\lambda > 0: \, \int_{X}\Phi(|f(x)|/\lambda)d\mu(x)\leq1\right\},
\end{equation*}
is finite. We can easily see here that if $\Phi(t)=t^p$ for $1\leq p<+\infty$, we have $L^\Phi(X)=L^p(X)$ and we recover the classical Lebesgue spaces. Note now that a Young function $\Phi$ is convex, increasing, left continuous, $\Phi(0)=0$ and $\Phi(t)\to\infty$ as $t\to +\infty$.
 Since the functional $\|\cdot\|_{L^\Phi}$ is a norm, we have $\|\lambda f\|_{L^\Phi}=|\lambda|\; \|f\|_{L^\Phi}$ and if $f,g$ are two measurable functions such that $|f|\leq |g|$ \emph{a.e.}, then we have the order-preserving property (\ref{Monotony})
$$\|f\|_{L^\Phi(X)}\leq \|g\|_{L^\Phi(X)}.$$
However, the property (\ref{Puissance}) should be handled more carefully, and for this, we will use the following rescaling property as defined in Section 3 of \cite{RaSam}: for any real $\rho>0$, we define the space $L^\Phi_\rho(X)$ by the condition
$$L^\Phi_\rho(X)=\{f:X\longrightarrow \mathbb{R}: \|f\|_{L^\Phi_\rho}<+\infty\},$$
where 
$$\|f\|_{L^\Phi_\rho}=\inf\left\{\lambda > 0: \, \int_{X}\Phi_\rho(|f(x)|/\lambda)d\mu(x)\leq1\right\},$$
with $\Phi_\rho(t)=\Phi(t^\rho)$. With this definition of the functional $\|\cdot\|_{L^\Phi_\rho}$ we have:
$$\||f|^\rho\|_{L^\Phi}=\|f\|_{L^\Phi_\rho}^\rho,$$
see Lemma 3.2 of \cite{RaSam} for a proof of this fact.\\

The last ingredient -the boundedness of the maximal operator- also requires a different treatment, and it is classical to impose some restrictions on the Young function $\Phi$. A function $\Phi$ satisfies the $\Delta_2$-condition if there exists a constant $C_\Phi>0$ such that 
$$\Phi(2\tau)\le C_\Phi \Phi(\tau),$$
for every $\tau\ge 0$. The $\Delta_2$-condition implies that $\Phi$ is strictly increasing and continuous. Now, a Young function $\Phi$ is said to satisfy the $\nabla_2$-condition, denoted also by $\Phi\in \nabla_2$, if
$$\Phi(\tau)\leq\frac{1}{2C} \Phi(C\tau),\qquad \tau\geq 0,$$
for some $C > 1$. With these two restrictions, if $\Phi$  satisfies the $\Delta_2$-condition and $\Phi\in  \nabla_2$ then we have the following boundedness property 
\begin{equation*}
\|\mathscr{M}_\mu(f)\|_{L^{\Phi}}\leq C\|f\|_{L^{\Phi}}.
\end{equation*}
See \cite{Cianchi0} or \cite{Heikkinen} for a proof of this fact. See also \cite{Cianchi1} or \cite{Derigoz} for more properties of this type of operators in Orlicz spaces on $\mathbb{R}^n$.\\

With all these ingredients at hand, we obtain from the pointwise inequality (\ref{PointWiseIneq_UpperGradient}):
$$\|T^*_K(f)\|_{L^\Phi}\leq C\left\|\mathscr{M}_\mu(g)^{1-\frac{q}{\nu}}\right\|_{L^\Phi}\|g\|_{\mathcal{M}^{p,q}_\mu}^{\frac{q}{\nu}}.$$
Using the space $L^\Phi_\rho$ introduced above with $\rho=1-\frac{q}{\nu}$ we can write
$$\|T^*_K(f)\|_{L^\Phi}\leq C\|\mathscr{M}_\mu(g)\|_{L^\Phi_{(1-\frac{q}{\nu})}}^{1-\frac{q}{\nu}}\|g\|_{\mathcal{M}^{p,q}_\mu}^{\frac{q}{\nu}}.$$
Now, if the function $\Phi_{(1-\frac{q}{\nu})}$ satisfies the $\nabla_2$-condition and if we have $\Phi_{(1-\frac{q}{\nu})}\in\Delta_2$, then we obtain 
$$\|T^*_K(f)\|_{L^\Phi}\leq C\|g\|_{L^\Phi_{(1-\frac{q}{\nu})}}^{1-\frac{q}{\nu}}\|g\|_{\mathcal{M}^{p,q}_\mu}^{\frac{q}{\nu}}.$$

\item {\bf Lebesgue spaces of variable exponent}. To introduce the Lebesgue spaces of variable exponent we first need to consider a measurable function $p:X\longrightarrow [1,+\infty]$, we then define $p^-=\underset{x\in X}{\mbox{ess inf}} \; \{p(x)\}$ and $p^+=\underset{x\in X}{\mbox{ess sup}} \; \{p(x)\}$ and, for the sake of simplicity and to avoid technicalities, we will always assume here that we have $1<p^-\leq p^+<+\infty$. Next, for $f:X\longrightarrow \mathbb{R}$ a $\mu$-measurable function we define the modular function $\varrho_{p(\cdot)}$ associated with the variable exponent $p(\cdot)$ by the expression
\begin{equation*}
\varrho_{p(\cdot)}(f)=\int_{X}|f(x)|^{p(x)}d\mu(x),
\end{equation*}
and we will consider the following Luxemburg norm
$$\|f\|_{L^{p(\cdot)}}=\inf\{\lambda > 0: \, \varrho_{p(\cdot)}(f/\lambda)\leq1\}.$$
We will thus define the spaces $L^{p(\cdot)}(X)$ as the set of $\mu$-measurable functions $f$ such that the quantity $\|f\|_{L^{p(\cdot)}}$ is finite. The spaces $L^{p(\cdot)}(X)$ are normed spaces and they have some nice structural properties (see the books \cite{Cruz1}, \cite{Diening} for more details). In particular the Luxemburg norm is order preserving: if $f,g\in L^{p(\cdot)}(X)$ are such that $|f|\leq |g|$ \emph{a.e.}, then we have 
$$\|f\|_{L^{p(\cdot)}}\leq \|g\|_{L^{p(\cdot)}},$$
see \cite[Proposition 2.7]{Cruz1}.
Another particular feature of the Luxemburg norm for Lebesgue spaces of variable exponent is the following: for all real parameter $\rho>0$ such that $\frac{1}{p^-}\leq \rho<+\infty$, we have the identity
$$\||f|^\rho\|_{L^{p(\cdot)}}=\|f\|^\rho_{L^{\rho p(\cdot)}}.$$
See \cite[Proposition 2.18]{Cruz1} for a proof of this fact. For a more detailed study of these spaces, see the books \cite{Cruz1} and \cite{Diening}.\\

We thus have in this setting the properties (\ref{Monotony}) and (\ref{Puissance}), however, some extra assumptions on the function $p(\cdot)$ are needed to obtain the boundedness of the maximal operator $\mathscr{M}_\mu$ and for this we need to recall some definitions in order to write the condition which gives the boundedness of the maximal operator $\mathscr{M}_\mu$. Indeed, let $r(\cdot):X\longrightarrow [0,+\infty[$ be a measurable function.
\begin{itemize}
\item We say that $r(\cdot)$ is \emph{locally log-H\"older continuous} and write $r(\cdot)\in LH_0$ if there exists a constant $C_0$ such that 
$$|r(x)-r(y)|\leq\frac{C_0}{-\log(d(x, y))},$$
for all $x,y\in X$ with $d(x, y)< \frac{1}{2}$.\\
\item We say that $r(\cdot)$ is \emph{locally log-H\"older continuous at infinity with respect to a base point $x_0\in X$} and  we write $r(\cdot)\in LH_\infty$ if there exist two constants $C_\infty,r_\infty$ such that 
$$|r(x)-r_\infty|\leq\frac{C_\infty}{\log(e+d(x, x_0))},$$
for all $x \in X$.\\
\item For $r(\cdot)\in LH_0\cap LH_\infty$ we say that $r(\cdot)$ is \emph{globally log-H\"older continuous} and we define $LH=LH_0\cap LH_\infty$.
\end{itemize}  
With these notions, we obtain the following condition: if $p(\cdot):X\longrightarrow [1,+\infty]$ is such that $1/p(\cdot)\in LH$ then the maximal function $\mathscr{M}_\mu$ is bounded in the Lebesgue spaces of variable exponents:
$$\|\mathscr{M}_\mu(f)\|_{L^{p(\cdot)}}\leq C\|f\|_{L^{p(\cdot)}}.$$
See Corollary 1.8 from \cite{Adamowicz} and \cite[Theorem 1.1]{Cruz-Shukla} in the setting of spaces of homogeneous type for a proof of this estimate.\\
	
Now, with all these results at our disposal, under the framework of the Theorem \ref{Theo_Pointwise_OpMorrey}, if we apply the $L^{r(\cdot)}$ norm to the inequality (\ref{PointWiseIneq_UpperGradient}), we obtain
$$\|T^*_K(f)\|_{L^{r(\cdot)}}\leq C\left\|\mathscr{M}_\mu(g)^{1-\frac{q}{\nu}}\right\|_{L^{r(\cdot)}}\|g\|_{\mathcal{M}^{p,q}_\mu}^{\frac{q}{\nu}}=C\left\|\mathscr{M}_\mu(g)\right\|_{L^{(1-\frac{q}{\nu})r(\cdot)}}^{1-\frac{q}{\nu}}\|g\|_{\mathcal{M}^{p,q}_\mu}^{\frac{q}{\nu}},$$
thus if $1<(1-\frac{q}{\nu})r^-\leq (1-\frac{q}{\nu})r(\cdot)\leq (1-\frac{q}{\nu})r^+<+\infty$ and if $\frac{1}{(1-\frac{q}{\nu})r(\cdot)}\in LH$, then, by the boundedness of the maximal function in the Lebesgue spaces of variable exponent, we can write
\begin{eqnarray*}
\|T^*_K(f)\|_{L^{r(\cdot)}}\leq C\left\|g\right\|_{L^{(1-\frac{q}{\nu})r(\cdot)}}^{1-\frac{q}{\nu}}\|g\|_{\mathcal{M}^{p,q}_\mu}^{\frac{q}{\nu}}.
\end{eqnarray*}
\begin{Remarque}
Note that since in the statement of the Theorem 1 we need the condition \eqref{ConditionConstantes}, it is clear that $\nu>1$, therefore we can apply a result from \cite{Cruz-Shukla} namely Corollary 1.5: If $(X,d,\mu)$ is an Ahlfors regular space with dimension $1<\nu$, let $p(\cdot):X\longrightarrow [1,+\infty]$ be such that $p(\cdot)\in LH$ and $1<p_-\le p_+<\nu$. Define $q(\cdot)$ by $\frac{1}{p(\cdot)}-\frac{1}{q(\cdot)}=\frac{1}{\nu}$. Then there exists $C=C(p(\cdot),\nu, X)$ such that 
$$\|\mathcal{R}_{1,\mu}\left(g\right)\|_{L^{q(\cdot)}}\le C\|g\|_{L^{p(\cdot)}}.$$
Using this estimate and our pointwise estimate \eqref{EstimationOperateur1}, we obtain that
$$\|T^*_K(f)\|_{L^{q(\cdot)}}\leq \|\mathcal{R}_{1,\mu}\left(g\right)\|_{L^{q(\cdot)}}\le C\|g\|_{L^{p(\cdot)}},$$
which is a useful functional inequality that only involves Lebesgue spaces of variable exponent. 
\end{Remarque}
\item {\bf Orlicz-Musielak spaces}

Let us first recall the definition of these spaces.  Let $\varphi(x,t):X\times [0,+\infty[\longrightarrow [0,+\infty[$ be a function satisfying the folllowing conditions:
\begin{itemize}
\item[$(\varphi 1)$] The function $\varphi(\cdot,t)$ is measurable on $X$ for every $t\ge 0$ and $\phi(x,\cdot)$ is continuous on $[0,+\infty[$ for every $x\in X$;\\
\item[$(\varphi 2)$] there exists a constant $C_1\ge 1$ such that $\frac{1}{C_1}\le\varphi(x,1)\le C_1$ for every $x\in X$;\\
\item[$(\varphi 3)$] there exists a constant $C_2\ge 1$ such that $\frac{\varphi(x,t_1)}{t_1}\le C_2\frac{\varphi(x,t_2)}{t_2}$ for all $x\in X$ and $0<t_1<t_2$.\\ 
\end{itemize}
We define now $\bar{\varphi}(x,t)=\underset{0<s\le t}{\sup} \varphi(x,s)$ and we write $\Phi(x,t)=\displaystyle{\int_0^t\frac{\bar\varphi(x,s)}{s}dr}$ for all $x\in X$ and $t\ge 0$. For a given function $\varphi(x,t)$ satisfying the conditions $(\varphi 1), (\varphi 2)$ and $(\varphi 3)$, the Musielak-Orlicz space $L^\varphi(X)$ is a Banach space with the following norm (see \cite{RaSam} and the references therein):
$$\|f\|_{L^\varphi}=\inf\left\{\lambda>0:\int_X \Phi\left(y,\frac{|f(y)|}{\lambda}\right)d\mu(y)\le 1\right\}.$$
Since the functional $\|\cdot\|_{L^\varphi}$ is a norm, as in the case of Orlicz spaces, we have that $\|\lambda f\|_{L^\varphi}=|\lambda|\; \|f\|_{L^\varphi}$ for any $\lambda \in \mathbb{R}$ and if $f,g$ are two measurable functions such that $|f|\leq |g|$ \emph{a.e.}, then we have the order-preserving property (\ref{Monotony})
$$\|f\|_{L^\varphi}\leq \|g\|_{L^\varphi}.$$
For the property (\ref{Puissance})  we  use, as in the case of Orlicz spaces, the following rescaling property as defined in Section 3 of \cite{RaSam}: for any real $\rho>0$, we define the space $L^\varphi_\rho(X)$ by the condition
$$L^\varphi_\rho(X)=\{f:X\longrightarrow \mathbb{R}: \|f\|_{L^\varphi_\rho}<+\infty\},$$
where 
$$\|f\|_{L^\varphi_\rho}=\inf\left\{\lambda > 0: \, \int_{X}\Phi_\rho(y,|f(y)|/\lambda)d\mu(y)\leq1\right\},$$
with $\Phi_\rho(x,t)=\Phi(x,t^\rho)$. With this definition of the functional $\|\cdot\|_{L^\varphi_\rho}$ we have:
$$\||f|^\rho\|_{L^\varphi}=\|f\|_{L^\varphi_\rho}^\rho,$$
see Lemma 3.2 of \cite{RaSam}.\\

Now, for the boundedness of the maximal operator, we need to consider some conditions for the function $\varphi(x,t)$. Let us consider $p,q\ge 1$ and $\eta>0$ be given.\\
\begin{itemize}
\item[$(\varphi 3;0;p)$] there exists a constant $A_{2,0,p}\ge 1$ such that
$$\frac{\varphi(x,t_1)}{t_1^{p}}\le A_{2,0,p}\frac{\varphi(x,t_2)}{t_2^{p}},$$
for all $x\in X$ whenever $0<t_1<t_2\le 1;$\\
\item[$(\varphi 3;\infty;q)$]  there exists a constant $A_{2,\infty,q}\ge 1$ such that
$$\frac{\varphi(x,t_1)}{t_1^{q}}\le A_{2,\infty,q}\frac{\varphi(x,t_2)}{t_2^{q}},$$
for all $x\in X$ whenever $1\le t_1<t_2;$\\
\item[$(\varphi 5;\eta)$] for every $\gamma>0$ there exists a constant $B_{\gamma,\eta},$ $0<B_{\gamma,\eta}\le 1$ such that 
$$\varphi(x, B_{\gamma,\eta}t)\le \varphi(y,t),$$
whenever $x,y\in X$, $d(x,y)\le \frac{\gamma}{t^\eta}$ and $t\ge 1$;\\
\item[$(\varphi 6)$] there exist two functions on $X$, $g$ and $h$ and there is a constant $B_\infty$, $0<B_\infty\le 1$ such that $0\le g(x)\le 1$, $0\le h(x)\le 1$ for every $x\in X$ and $\varphi(\cdot,g(\cdot))\in L^1(X)$, $h\in L^1(X)$ and 
$$\varphi(x,B_\infty t)\le \varphi(y,t)+h(x),$$
for $x,y\in X$ and $g(y)\le t\le 1$.\\
\end{itemize}
The previous conditions are enough to obtain the boundedness of the maximal functions in this framework. Indeed, since in our setting the measure $\mu$ is Ahlfors $\nu$-regular implies that $\mu$ satisfies the doubling condition. It is known that if $\mu$ satisfies the doubling condition then the space $X$ satisfies the $(M\lambda)$ condition for any $\lambda>0$ \emph{i.e.} there exists a constant $C>0$ such that 
$$\mu(\{x\in X:M_\lambda f(x)>k\})\le \frac{C}{k}\int_X|f(y)|d\mu(y),$$
for all measurable functions $f\in L^1(X)$ and $k>0$, where $M_\lambda f$ is the dilated Hardy-Littlewood maximal function defined by
$$M_\lambda f(x)=\sup_{r>0}\frac{1}{\mu(B(x,\lambda r))}\int_{B(x,r)}|f(y)d\mu(y)|,$$
see \cite{{GCRdF}} for more details. In particular for $\lambda=1$, $M_1=\mathscr{M}_\mu$ the Hardy-Littlewood maximal operator satisfies $(M1)$.\\

Next suppose that the function $\varphi(x,t)$ satisfies the conditions $(\varphi3;0,p)$, $(\varphi3;\infty,q)$, $(\varphi 5;\eta)$ and $(\varphi 6)$ stated above for $p_1,q_1>1$ and $\eta>0$ satisfying $\eta\le\frac{q}{\nu}$. Then there is a constant $C>0$ such that 
$$\|\mathscr{M}_\mu(f)\|_{L^\varphi}\le C \|f\|_{L^\varphi},$$
for all $f\in L^\varphi(X)$, see \cite[Theorem 3.7 ]{OS25} for a proof of this fact. See also \cite{Ohno}.\\

Now from the pointwise inequality (\ref{PointWiseIneq_UpperGradient}), we obtain that:
$$\|T^*_K(f)\|_{L^\varphi}\leq C\left\|\mathscr{M}_\mu(g)^{1-\frac{q}{\nu}}\right\|_{L^\varphi} \|g\|_{\mathcal{M}^{p,q}_\mu}^{\frac{q}{\nu}}.$$
Using the space $L^\varphi_\rho(X)$ introduced above with $\rho=1-\frac{q}{\nu}$ we can write
$$\|T^*_K(f)\|_{L^\varphi}\leq C\|\mathscr{M}_\mu(g)\|_{L^\varphi_{(1-\frac{q}{\nu})}}^{1-\frac{q}{\nu}}\|g\|_{\mathcal{M}^{p,q}_\mu}^{\frac{q}{\nu}}.$$
Now, if the function $\varphi_{(1-\frac{q}{\nu})}$ satisfies the conditions $(\varphi3;0,p)$, $(\varphi3;\infty,q)$, $(\varphi 5;\eta)$ and $(\varphi 6)$ for $p,q>1$ and $\eta>0$ satisfying $\eta\le\frac{q}{\nu}$, then we obtain the following functional inequality
$$\|T^*_K(f)\|_{L^\varphi}\leq C\|g\|_{L^\varphi_{(1-\frac{q}{\nu})}}^{1-\frac{q}{\nu}}\|g\|_{\mathcal{M}^{p,q}_\mu}^{\frac{q}{\nu}}.$$
\end{itemize}
We studied here Lebesgue spaces, Lorentz, Morrey, Orlicz, Lebesgue spaces of variable exponent and Orlicz-Musielak spaces. This list is of course non exhaustive and any functional space $E^\mathfrak{r}(X)$ endowed with a norm $\|\cdot\|_{E^\mathfrak{r}}$ that satisfies the properties (\ref{Monotony}), (\ref{Puissance}) and (\ref{GeneralMaximal}) we will lead to a generic estimate of the form (\ref{GeneralFuncIneq}).\\

\noindent {\bf Conflict of interest.} We declare that we do not have any commercial or associative interest that represents a conflict of interest in connection with the work submitted.\\

\noindent {\bf Acknowledgment.} This work was supported by the GDRI ECO-Math.

\end{document}